\documentclass[a4paper,12pt,reqno]{amsart}
\usepackage{amsmath}
\newtheorem{thm}{Theorem}[section]
\newtheorem{lemma}[thm]{Lemma}
\newtheorem{prop}[thm]{Proposition}

\newtheorem{thmintro}{Theorem}
\theoremstyle{remark}

\newtheorem*{rem}{Remark}
\newtheorem*{acknowledgement}{Acknowledgment}

\numberwithin{equation}{section}

\newcommand{\C}{\mathbb{C}}

\newcommand{\N}{\mathbb{N}}
\newcommand{\Z}{\mathbb{Z}}
\newcommand{\R}{\mathbb{R}}
\newcommand{\eps}{\varepsilon}

\def\dim{\operatorname{dim}}
\def\diam{\operatorname{diam}}
\newcommand{\Ima}{\operatorname{Im}}
\newcommand{\Rea}{\operatorname{Re}}
\newcommand{\Unb}{{\rm{Unb}}}
\newcommand{\Esc}{{\rm{Esc}}}
\newcommand{\abs}[1]{\left|#1\right|}
\begin{document}

\title[Escape rate and Hausdorff measure]{Escape rate
and Hausdorff measure for entire functions}
\thanks{Both authors are supported by the Deutsche
Forschungsgemeinschaft, grant no. Be 1508/7-2.}
\author{Walter Bergweiler}
\author{J\"{o}rn Peter}
\address{Mathematisches Seminar, Christian-Albrechts-Universit\"{a}t zu Kiel,
 24118 Kiel, Germany}
\begin{abstract}
The escaping set of an entire function is the set of points that tend
to infinity under iteration. We consider subsets of the escaping set
defined in terms of escape rates and obtain upper and lower bounds
for the Hausdorff measure of these sets with respect to certain gauge functions.
%\keywords{Escaping set, Julia set, Hausdorff measure, Hausdorff dimension, iterated function scheme, limit set, Eremenko-Lyubich class}
\end{abstract}

\maketitle

\section{Introduction and results} \label{sec1}

The \emph{escaping set} $I(f)$ of a non-linear and non-constant entire function $f$ consists of 
all points in the complex plane $\C$ which tend to $\infty$ under iteration 
of~$f$. The \emph{Julia set} $J(f)$ is the subset of
$\C$ where the iterates fail to be normal. By a result of 
Eremenko \cite{e} we have $J(f)=\partial I(f)$.
For an introduction to the dynamics of entire functions we refer to~\cite{Bergweiler93}.

Considerable attention has been paid to the Hausdorff dimension and 
Hausdorff measures of these sets; see \cite{Stallard08} for a survey.
We denote the Hausdorff dimension of a subset $A$ of $\C$ by $\dim A$ and
the Hausdorff measure of $A$ with respect to a gauge function $h$ by $H_h(A)$; 
see section~\ref{sec2} for these definitions and some discussion of them.

The first results on the Hausdorff dimension of Julia sets of transcendental entire
functions are due to McMullen who proved \cite[Theorem~1.2]{m} that $\dim J(\lambda e^z)=2$
for $\lambda\in\C$, $\lambda
\neq 0$. He notes \cite[p.~336]{m} that
$H_h(J(\lambda e^z))=\infty$ if
$h(t)=t^2\log\log\ldots\log(1/t)$, for any number of logarithms.
A refinement of this result is given in~\cite{Peter}.
McMullen actually showed that
$\dim I(\lambda e^z)=2$ and then noted that
$I(\lambda e^z)\subset J(\lambda e^z)$. 

McMullen's results initiated a large body of research on Hausdorff dimension
and measure of Julia sets and escaping sets;
see the survey \cite{Stallard08} as
well as more recent results in \cite{b,bkz,bks,Peter2,Rempe09,rs,s}.
Most of these results,
exceptions being \cite{bk,Stallard90}, have been concerned with the
\emph{Eremenko-Lyubich class} $B$ which consists of all transcendental entire
functions $f$ for which the set sing$(f^{-1})$ of singularities of
the inverse of $f$ is bounded. We note that sing$(f^{-1})$ coincides with the
set of
critical and finite asymptotic values of~$f$. For $f(z)=\lambda e^z$ we
have sing$(f^{-1})=\{0\}$ and thus $f\in B$. Eremenko and
Lyubich \cite[Theorem~1]{el} proved that if $f\in B$, then $I(f)\subset J(f)$ and thus
$J(f)=\overline{I(f)}$.

Baker \cite{b2} showed that if $f$ is transcendental entire, not necessarily in
the Eremenko-Lyubich class, then $J(f)$ contains continua.
Rippon and Stallard~\cite{rs4} proved that $I(f)$ also contains
continua.
In fact, they showed 
\cite[Theorem 1.3]{rs3}
that $I(f)\cap J(f)$ contains continua.
It follows that $\dim J(f)\geq 1$ and $\dim I(f)\geq 1$ for every transcendental entire function~$f$.
Bishop~\cite{Bishop} has recently constructed an entire transcendental function $f$ 
satisfying $\dim J(f)=1$. On the other hand,
Stallard (\cite{s2}, see also \cite{bkz} and
 \cite[Theorem~1.5]{Bergweiler08}) proved that if $f\in B$, then 
$\dim J(f)>1$.
Moreover, she showed \cite{s4} that for each $d\in(1,2]$ there exists a function $f\in B$ such that $\dim J(f)=\dim I(f)=d$.

In contrast to Stallard's result that $\dim J(f)>1$ for $f\in B$, Rempe and Stallard \cite{rs} showed that there exists a function $f\in B$ with $\dim I(f)=1$.
Thus $H_h(I(f))=0$ if $h(t)=t^{1+\eps}$ with $\eps>0$.

In recent studies of the escaping set, several subsets of the escaping set defined in terms of escape rates have turned out to be important; see, e.g., \cite{rs2,rs3}.
For example, the key idea in the proof 
in \cite{rs4} that $I(f)$ has at least one unbounded component
was to consider a subset $A(f)$ of $I(f)$ which, roughly speaking, consists of 
all points tending to $\infty$ at the fastest rate
compatible with the growth of $f$.
This set $A(f)$, now called the \emph{fast 
escaping set}, was introduced in \cite{bh} to study a problem concerning 
permutability. It is shown in~\cite{rs4} that all
components of $A(f)$ are unbounded continua.
In contrast, the set $I(f)\setminus A(f)$ need not contain continua.
In fact, 
if $f(z)=\lambda e^z$ with $0<\lambda<1/e$, then
$I(f)\setminus A(f)$ is totally disconnected; cf. \cite{Devaney84,Mayer90,rrs}.
Theorem~\ref{thm1} below implies that 
$\dim (I(f)\setminus A(f))\geq 1$ for $f\in B$.

For a sequence $(p_n)$ of positive real numbers tending to $\infty$ we denote 
by $\Esc(f,(p_n))$ the set of all $z\in I(f)$ such that 
$\abs{f^n(z)}\leq p_n$ for all large~$n$.
We prove that $H_h(\Esc(f,(p_n)))=\infty$ if $h(t)$ tends 
to $0$ slower than the functions $t\mapsto t^{1+\eps}$.

\begin{thmintro} \label{thm1}
Let $f\in B$ and let $h$ be a gauge function satisfying
\begin{equation} \label{1a}
\lim_{t\to 0}\frac{\log h(t)}{\log t}=1.
\end{equation}
Let $(p_n)$ be a sequence of positive real numbers tending to $\infty$. Then
$$H_h(\Esc(f,(p_n)))=\infty.$$
In particular, $H_h(I(f))=\infty$ and
$\dim \Esc(f,(p_n))\geq 1$.
\end{thmintro}

In contrast to Theorem~\ref{thm1}, we show that the set of points which 
escape with a definitive speed (in particular the fast escaping set) can 
have zero $H_h$-measure for certain gauge functions satisfying \eqref{1a}.

To formulate the result,
let again $(p_n)$ be a real sequence 
tending to~$\infty$ and
let $\Unb(f,(p_n))$ be the set of all 
$z\in\C$ such that $\abs{f^n(z)}> p_n $ for infinitely many~$n$.
Thus 
$$\Esc(f,(p_n))=I(f)\setminus \Unb(f,(p_n)).$$
Note that points in $\Unb(f,(p_n))$ have unbounded orbits, but need not be 
in $I(f)$.

\begin{thmintro}  \label{thm2}
Let $(p_n)$ be a real sequence tending to $\infty$ and let $h$ be a 
gauge function satisfying
\begin{equation} \label{1b}
\lim_{t\to 0}\frac{h(t)}{t}=0.
\end{equation}
Then there exists 
$f\in B$ such that
$$H_h(\Unb(f,(p_n)))=0.$$
\end{thmintro}
Note that \eqref{1b} cannot be weakened since $A(f)$
contains continua~\cite{rs4}.
Similarly, \eqref{1a} cannot be weakened by~\cite{rs}.

There exist gauge functions $h$ satisfying both \eqref{1a} and \eqref{1b}.
Thus Theorem~\ref{thm1} and Theorem~\ref{thm2} imply
that the set of slow escaping points can be larger (from 
a measure-theoretic point of view) than the fast escaping set.
This is in contrast to the situation for the functions $\lambda e^z$ where
the fast escaping set is larger~\cite{KU}.

We describe the idea of the proof of Theorem~\ref{thm1}.
An important ingredient is a result of Bara\'nski, Karpi\'nska and
Zdunik \cite{bkz} saying that
the hyperbolic dimension  of the Julia set of
a transcendental meromorphic function with a logarithmic tract
is strictly greater than~$1$.
In the proof
the authors construct iterated function schemes
consisting of inverse branches of the second iterate of $F$ which map
a certain square $S$ into itself. Here $F$ is a function obtained
from $f$ by a logarithmic change of variable
(cf.\ section~\ref{sec3}).
As the result is stated in~\cite{bkz} not quite in the form we need,
we formulate a version of it as Lemma~\ref{th34} below;
cf.\ the remark at the end of section~\ref{sec4} for the difference
to~\cite{bkz}.  
For the convenience
of the reader we also include the proof.

We will take a suitable sequence $(S_k)$ of such squares tending to infinity
and obtain a subset of $\Esc(f,(p_n))$ by considering points in $S_1$ which
stay in $S_1$ under many iterations of $F^2$, are then mapped to $S_2$,
stay there for many iterations of $F^2$, and so forth.

In order to estimate the Hausdorff measure of the set obtained, we extend 
in section~\ref{sec2} 
a result of \cite{f} concerning the
Hausdorff dimension of the limit set of iterated function schemes.
The definitions and basic properties of
Hausdorff measures and the Hausdorff dimension as well
as iterated function schemes are recalled at the beginning of that section.
The results of sections~\ref{sec2} and~\ref{sec3}
 are then combined in section~\ref{sec4} to prove Theorem~\ref{thm1}. 

The function $f$ constructed in Theorem~\ref{thm2} will be large inside
a narrow strip and bounded outside the strip.
In section~\ref{sec5} we will give preliminary results for the construction,
in particular we discuss some results of Ahlfors~\cite{Ahlfors30}
concerning conformal mappings of strips.
The actual proof of Theorem~\ref{thm2}  is then carried out in section~\ref{sec6}.

\section{Hausdorff measures} \label{sec2}
We recall the definition of Hausdorff measure and Hausdorff dimension; see
the book by Falconer \cite{f} for more details.

A \emph{gauge function} (or \emph{dimension function})
 is an increasing, continuous function
$h\colon [0,\eta)\to[0,\infty)$ which satisfies $h(0)=0$, where $\eta>0$.
For $A\subset\C$ and $\delta>0$ we call a sequence $(A_j)$ of
subsets of $\C$ a $\delta$-\emph{cover} of $A$ if $\diam A_j <\delta$ for all $j\in\N$ and
$$A\subset\bigcup_{j=1}^{\infty}A_j.$$
Here $\diam A_j$ denotes the (Euclidean) diameter of $A_j$. For a gauge
function $h$ we put
$$H_h^{\delta}(A)=\inf\left\{\sum_{j=1}^{\infty}h(\diam A_j)\colon (A_j)\text{ is }\delta\text{-cover of }A\right\}$$
and call
$$H_h(A)=\lim_{\delta\to 0}H_h^{\delta}(A)$$
the \emph{Hausdorff
measure} of $A$ with respect to the gauge function~$h$.
Note that $H_h^\delta(A)$ is a non-increasing function of $\delta$
so that the limit defining $H_h(A)$ exists. 
It is possible
that $H_h^\delta(A)=\infty$, meaning that $\sum_{j=1}^\infty h(\diam A_j)$
diverges for all $\delta$-covers $(A_j)$.
Similarly, we may have $H_h(A)=\infty$.
Given gauge functions $h_1,h_2$
and $A\subset\C$, we have
\begin{equation}\label{2j}
H_{h_1}(A)\leq H_{h_2}(A)
\quad\text{if } h_1(t)\leq h_2(t) \text{ for all small }t.
\end{equation}

In the special case that $h(t)=t^s$ for some $s>0$, we call $H_h(A)$ the
$s$-\emph{dimensional Hausdorff measure}. There exists $d\geq 0$ such that
$H_{t^s}(A)=\infty$ for $0<s<d$ and $H_{t^s}(A)=0$ for $s>d$. This value $d$ is
called the \emph{Hausdorff dimension} of $A$ and denoted by $\dim A$. 

The following result \cite[Theorem 7.6.1]{PrzUrb} about the Hausdorff measure 
is known as the mass distribution principle.
Here $D(x,r)$ denotes the open disk of radius $r$ around a point $x\in \C$.
\begin{lemma}\label{th21}
Let $h$ be a gauge function and $A\subset\C$. If there exists a Borel probability measure $\mu$ supported on $A$ such that
$$\lim_{r\to 0}\frac{\mu(D(x,r))}{h(r)}=0\quad \text{for all }x\in A,$$
then $H_h(A)=\infty$.
\end{lemma}
Let $D\subset\C$ be compact. A mapping $T\colon D\to D$ is called a \emph{contraction}
on $D$ if there exists $c\in(0,1)$ such that $\abs{T(z)-T(w)}\leq c\abs{z-w}$
for all $z,w\in D$. Clearly a contraction is continuous.
Let $T_1,\ldots,T_m$ be contractions on~$D$. Then the family
$T=\{T_1,\ldots,T_m\}$ is
called an \emph{iterated function scheme on D}. It can be shown (see \cite[Theorem 9.1]{f}) that there exists a unique non-empty compact set $X\subset D$ which is invariant for the $T_i$, that is, which satisfies
$$X=\bigcup_{i=1}^m T_i(X).$$
Moreover, we have
$$X=\bigcap_{n=1}^\infty\bigcup_{i_1,\ldots,i_n} 
\left( T_{i_1}\circ\ldots\circ T_{i_n}\right)(F)$$
for every non-empty compact set $F\subset D$ satisfying $T_i(F)\subset F$
for all~$i$. The set $X$ is called \emph{limit set} of the iterated function
scheme~$T$.

The following result \cite[Proposition~9.7]{f} gives a lower bound 
for the Hausdorff dimension of the limit set of an iterated function scheme;
see~\cite[Proposition~9.6]{f} for an analogous upper bound.
\begin{lemma}\label{th22}
Let $D\subset\C$ be compact and $T_1,\ldots,T_m\colon D\to D$ be contractions.
Suppose that $T_i(D)\cap T_j(D)=\emptyset$ for all $i\neq j$ and that for each $i$ there exists $b_i\in(0,1)$ with
\begin{equation}\label{formel99}
\abs{T_i(z)-T_i(w)}\geq b_i\abs{z-w}
\end{equation}
Define  $s$ by $$\sum_{i=1}^m b_i^{s}=1.$$
Then the limit set $X$ of the
iterated function scheme $T=\{T_1,\ldots,T_m\}$ satisfies $\dim X \geq s$.
\end{lemma}

We need a version of Lemma \ref{th22} for sequences of iterated function schemes. Let $(T_k)_{k\in\N}$ be such a sequence, with $T_k=\{T_{k,1},\ldots,T_{k,m_k}\}$. Define the \emph{limit set} $X$ of $(T_k)$ by 
$$X=\bigcap_{k=1}^\infty\bigcup_{\substack{1\leq i_l\leq m_l,\\1\leq l\leq k}}(T_{1,i_1}\circ\ldots\circ T_{k,i_k})(D).$$

\begin{lemma}\label{lemma1}
Let $D\subset\C$ be compact,
$(R_k)$ a sequence of iterated function schemes on $D$, with
$R_k=\{R_{k,1},\ldots,R_{k,m_k}\}$,
and
$h$ a gauge function satisfying~\eqref{1a}.
Let $X_k$ be the limit set of $R_k$. Suppose that, for all $k\in\N$,
\begin{enumerate}
\item[$(1)$] $R_{k,i}(D)\cap R_{k,j}(D)=\emptyset$ for $i\neq j$;
\item[$(2)$] for $1\leq j\leq m_k$ there exists $\tilde{b}_{k,j}\in (0,1)$ with 
$$
\sum_{j=1}^{m_k}\tilde{b}_{k,j}>1
\quad \text{and} \quad 
\abs{R_{k,j}(z)-R_{k,j}(w)}\geq\tilde{b}_{k,j}\abs{z-w}
\ \text{for}\ z,w\in D.
$$
\end{enumerate}
Then there exists an increasing sequence $(n_i)$ of positive integers 
such that if
$(T_k)$ is a sequence of iterated function schemes with 
\begin{equation}\label{formel2}
T_k\in\{R_1,\ldots,R_i\}
\text{ for all }k\leq n_i 
\end{equation}
and $X$ is its limit set, then $H_h(X)=\infty$.
\end{lemma}

\begin{rem}
It is clear that the conclusion also holds for
any increasing sequence $(n_i')$ satisfying $n_i'\geq n_i$ 
for all~$i$.
We will use this fact later.
\end{rem}

\begin{proof}
Let $h(t)=t^{1+\eps(t)}$. By \eqref{1a} we have $\eps(t)\to 0$ as $t\to 0$. We may assume that $\eps$ is non-decreasing, since otherwise we can replace $\eps(t)$ by $\max_{0\leq s\leq t}\eps(s)$ and use \eqref{2j}. Similarly, we may assume that $\eps(t)\to 0$ as $t\to 0$ so slowly that with $g(t)=t^{\eps(t)}$ and hence $h(t)=tg(t)$ we have 
\begin{equation}\label{formel6}
g(t)\to 0\quad\text{ as }t\to 0.
\end{equation}

By condition (2), the solution $\tilde{s}_k$ of $\sum_j\tilde{b}_{k,j}^{\tilde{s}_k}=1$ satisfies $\tilde{s}_k>1$ for all $k$.
Let an increasing sequence $(n_i)$ be given. We will show
 that $(n_i)$ has the required properties if it tends to infinity sufficiently fast. Let $(T_k)$ be a sequence of iterated function schemes satisfying \eqref{formel2} and let $X$ be its limit set. If $T_k=R_m$, we set $b_{k,j}=\tilde{b}_{m,j}$ and $s_k=\tilde{s}_m$.

By a standard argument, it is possible to  
construct a probability measure $\mu$ supported on $X$ such that 
$$\mu((T_{1,j_1}\circ\ldots\circ T_{k,j_k})(D))=b_{1,j_1}^{s_1}\cdot\ldots\cdot b_{k,j_k}^{s_k}.$$ 
Note that since $\sum_j b_{k,j}^{s_k}=1$ for all $k$, we have $$\mu\left(\bigcup_{j_1,\ldots,j_k}(T_{1,j_1}\circ\ldots\circ T_{k,j_k})(D)\right)=1.$$ 

Assuming that $(n_i)$ tends to $\infty$ sufficiently fast, we will
show that if $x\in X$ and $r>0$ is sufficiently small, then there exist
$k\in\N$ and $p_1,\ldots,p_k$ such that 
\begin{equation}\label{formel5}
D(x,r)\cap X\subset(T_{1,p_1}\circ\ldots\circ T_{k,p_k})(D)
\end{equation}
and
\begin{equation}\label{formel5a}
\mu((T_{1,p_1}\circ\ldots\circ T_{k,p_k})(D))\leq r^{1+2\eps(r)}.
\end{equation}

Let $$d_k=\min_{l\leq k}\min_{j\neq j'}\text{ dist}(T_{l,j}(D),T_{l,j'}(D)).$$ 
Rescaling if necessary, we may assume that $d_k<1$ for all $k$.
Further, define 
$$\alpha_k=\min_{l\leq k}\min_j b_{l,j},\quad\beta_k=\max_{l\leq k}\max_j b_{l,j},
\quad
\gamma_k=\min_{l\leq k}s_l,\quad\delta_k=\max_{l\leq k}s_l.$$ 
It follows from these definitions that the quantities $\alpha_{n_i},\beta_{n_i},\gamma_{n_i}$ and $\delta_{n_i}$ depend only on $R_1,\ldots,R_i$ and not on the choice of the sequence $(n_i)$. So we can choose $(n_i)$ such that
\begin{equation}\label{formel3}
\gamma_{n_{i+1}}-\frac{\delta_{n_{i+1}}\log(\alpha_{n_{i+1}}d_{n_{i+1}})}{\log(d_{n_i}(\beta_{n_i})^{n_i})}\geq 1+2\eps(d_{n_i}(\beta_{n_i})^{n_i}).
\end{equation}
To see that this condition can be satisfied we note that $\gamma_{n_{i+1}}>1$, the denominator tends to $-\infty$ if $n_i\to\infty$ and the right side tends to 1 as $n_i\to\infty$. The latter statements follow since $(\beta_{n_i})^{n_i}$ decreases to 0 as $n_i\to\infty$ (note that $\beta_{n_i}<1$).

Now let $r>0$ be small and choose $k\in\N$ with $$d_{k+1}\prod_{l=1}^{k+1} b_{l,p_l}\leq r\leq d_k\prod_{l=1}^{k} b_{l,p_l},$$ where $p_1,\ldots,p_{k+1}$ are the uniquely determined positive integers such that $x\in(T_{1,p_1}\circ\ldots\circ T_{k+1,p_{k+1}})(D)$.

If $(q_1,\ldots,q_k)\neq(p_1,\ldots,p_k)$ and $y=(T_{1,q_1}\circ\ldots\circ T_{k,q_k})(w)$ for some $w\in D$, we obtain
 with $k_0=\min\{l\in\{1,\ldots,k\}\colon q_l\neq p_l\}$ that
$$\abs{y-x}\geq b_{1,p_1}\cdot\ldots\cdot b_{k_0-1,p_{k_0-1}}\cdot d_{k_0}\geq d_k\prod_{l=1}^k b_{l,p_l}\geq r,$$ so \eqref{formel5} holds.

Let $i$ be such that $n_i\leq k<n_{i+1}$. 
By the choice of $k$ we have
$r\leq d_{n_i}(\beta_{n_i})^{n_i}$. Since $\eps$ is a non-decreasing function, \eqref{formel3} implies that
$$\gamma_{n_{i+1}}-\frac{\delta_{n_{i+1}}\log(\alpha_{n_{i+1}}d_{n_{i+1}})}{\log r}\geq 1+2\eps(r).$$ Multiplying this by $\log r$ and taking exponentials, it follows that
$$r^{\gamma_{n_{i+1}}}\frac{1}{(\alpha_{n_{i+1}}d_{n_{i+1}})^{\delta_{n_{i+1}}}}\leq r^{1+2\eps(r)}.$$ By the definition of the quantities above, we have
$$r^{\gamma_{k+1}}\leq r^{\gamma_{n_{i+1}}},\quad d_{k+1}^{-\gamma_{k+1}}\leq d_{n_{i+1}}^{-\delta_{n_{i+1}}}\quad\text{and}\quad b_{k+1,p_{k+1}}^{-s_{k+1}}\leq\alpha_{n_{i+1}}^{-\delta_{n_{i+1}}}.$$ 
This implies that
$$
\left(\frac{r}{d_{k+1}}\right)^{\gamma_{k+1}}\cdot b_{k+1,p_{k+1}}^{-s_{k+1}}\leq r^{1+2\eps(r)}
$$
so that 
\begin{align*}
&\mu((T_{1,p_1}\circ\ldots\circ T_{k,p_k})(D))
=\prod_{l=1}^k b_{l,p_l}^{s_l}
=
 b_{k+1,p_{k+1}}^{-s_{k+1}}
\prod_{l=1}^{k+1} b_{l,p_l}^{s_l}
\\ &
\quad \leq
 b_{k+1,p_{k+1}}^{-s_{k+1}}
\left(\prod_{l=1}^{k+1}b_{l,p_l}\right)^{\gamma_{k+1}}
\leq
b_{k+1,p_{k+1}}^{-s_{k+1}}
\left(\frac{r}{d_{k+1}}\right)^{\gamma_{k+1}}
\leq r^{1+2\eps(r)},
\end{align*}
which is \eqref{formel5a}.

Using \eqref{formel5} and \eqref{formel5a} we obtain 
$\mu(D(x,r))\leq r^{1+2\eps(r)}$ and thus
$$\frac{\mu(D(x,r))}{h(r)}=\frac{\mu(D(x,r))}{r^{1+\eps(r)}}\leq r^{\eps(r)}=g(r)\to 0\text{ as }r\to 0$$ by \eqref{formel6}. 
As this holds for all $x\in X$, Lemma \ref{th21} implies that $H_h(X)=\infty$.
\end{proof}

The next lemma follows easily from the previous one.

\begin{lemma}\label{th23}
Let $h$ be a gauge function satisfying \eqref{1a}.
Let
\begin{equation}\label{defD}
D=\{z\in\C\colon 0\leq\Rea z\leq 1,\; 0\leq\Ima z\leq 1\}
\end{equation}
and let $(P_k)_{k\in\N}$ and $(Q_k)_{k\in\N}$ be sequences of
iterated function schemes on~$D$, with
$$P_k=\{P_{k,1},\ldots,P_{k,l_k}\}
\quad\text{and}\quad
Q_k=\{Q_{k,1},\ldots,Q_{k,m_k}\}.$$
Suppose that
\begin{enumerate}
\item[$(1)$]
$P_{k,i}(D)\cap P_{k,j}(D)=\emptyset$ for all
$k\in\N$ and $i\neq j$.
\item[$(2)$] For $k\in\N$,  $1\leq i\leq l_k$ and $1\leq j\leq m_k$, there exist 
$\tilde{b}_{k,i},\overline{b}_{k,j}\in(0,1)$ with
\begin{equation}\label{2a}
\sum_{i=1}^{l_k}\tilde{b}_{k,i}>1
\end{equation}
such that for $z,w\in D$, we have 
$\tilde{b}_{k,i}\abs{z-w}\leq\abs{P_{k,i}(z)-P_{k,i}(w)}$ and $\overline{b}_{k,j}\abs{z-w}\leq\abs{Q_{k,j}(z)-Q_{k,j}(w)}.$
\end{enumerate}
Then there exists an increasing sequence $(n_i)$ of positive integers 
such that with $n_0=0$ the limit set $X$ of
the sequence $(T_k)$ of iterated function schemes defined by
\begin{equation}\label{2b}
T_k=
\begin{cases}
P_i& \text{for } n_{i-1}+1\leq k\leq n_i-1,\\
Q_i& \text{for } k=n_i,
\end{cases}
\end{equation}
satisfies $H_h(X)=\infty$.
\end{lemma}

\begin{proof}
We have not assumed that $Q_k$ satisfies condition (1) of Lemmas~\ref{lemma1} and~\ref{th23}.
But we may do so since otherwise we could work with $Q_k=\{Q_{k,1}\}$ so that
$m_k=1$, as this would make $X$ smaller. In contrast,
$l_k\geq 2$ by~\eqref{2a}.

For $k\in\N$ we put $R_{2k-1}=P_k$ and $R_{2k}=(P_k)^{p_k}\circ Q_k$, where
$p_k\in\N$. Then the $R_{2k-1}$ satisfy the hypotheses of Lemma~\ref{lemma1}.
We claim that the $R_{2k}$ also satisfy these hypotheses if the 
$p_k$ are chosen large enough.
The conclusion then follows from Lemma~\ref{lemma1}.

Let $1\leq j\leq m_k$ and $q_1,\ldots,q_{p_k}\in\{1,\ldots,l_k\}$. Then
\begin{align*}
\abs{P_{k,q_{p_k}}\circ\ldots\circ P_{k,q_1}\circ Q_{k,j}(z)-P_{k,q_{p_k}}\circ\ldots\circ P_{k,q_1}\circ Q_{k,j}(w)}\\
\geq\overline{b}_{k,j}\prod_{n=1}^{p_k}\tilde{b}_{k,q_n}\abs{z-w}
=:c_{j,q_1,\ldots,q_{p_k}}\abs{z-w}.
\end{align*}
Since 
\begin{align*}
&\sum_{j=1}^{m_k}\sum_{q_1,\ldots,q_{p_k}\in\{1,\ldots,l_k\}}
c_{j,q_1,\ldots,q_{p_k}}
=\sum_{j=1}^{m_k}\sum_{q_1,\ldots,q_{p_k}\in\{1,\ldots,l_k\}}
\left(\overline{b}_{k,j}\prod_{n=1}^{p_k}\tilde{b}_{k,q_n}\right)\\
& \quad \geq m_k\min_j\overline{b}_{k,j}
\sum_{q_1,\ldots,q_{p_k}\in\{1,\ldots,l_k\}}\prod_{n=1}^{p_k}\tilde{b}_{k,q_n}
=m_k\min_j\overline{b}_{k,j}
\left(\sum_{i=1}^{l_k}\tilde{b}_{k,i}\right)^{p_k}>1
\end{align*}
by \eqref{2a} if $p_k$ is large enough, the above claim follows.  
\end{proof}

Note that 
Lemma \ref{th22} implies that dim $X>s$ whenever $\sum_{i=1}^m b_i^s>1$.
The converse is not true in general, but it holds for some iterate of $T$ if
the maps in $T$ are conformal on a domain containing~$D$.
Thus in this case condition~(2) in Lemmas \ref{lemma1} and \ref{th23}
is essentially equivalent to the statement
 that the corresponding iterated function schemes have limit sets of
dimension greater than~$1$.

We prove this result for completeness.

\begin{prop}
Let $D\subset\C$ be compact. Let $T_1,\ldots,T_m\colon D\to D$ be contractions that extend conformally to self-maps of some domain $B$ that 
contains $D$. 
Suppose 
that the limit set $X$ of the iterated function scheme $T=\{T_1,\ldots,T_m\}$ satisfies dim $X>s$. 
Then, for large $p\in\N$, the iterated function scheme $$S=T^p=\{T_{i_p}\circ\ldots\circ T_{i_1}\colon i_1,\ldots,i_p\in\{1,\ldots,m\}\}$$ 
has the property that, for $i=(i_1,\dots,i_p)\in \{1,\ldots,m\}^p$ and $S_i=T_{i_p}\circ\ldots\circ T_{i_1}$,
\begin{equation}\label{formel1}
\abs{S_i(z)-S_i(w)}\geq b_i\abs{z-w}
\ \text{for} \ z,w\in D,
\end{equation}
with constants $b_i\in(0,1)$ satisfying $\sum_i b_i^s>1$.
\end{prop}

\begin{proof}
It is clear that $S=T^p$ and $T$ have the same limit set $X$ for all $p\in\N$. 
As the $T_k$ extend to conformal self-maps of $B$, the same is true 
for the $S_i$. 
Thus for all $i$ there exist some constant  $b_i\in(0,1)$ such
that~\eqref{formel1} holds. We define $b_i$ as the largest
number with this property.
It can be deduced from the Koebe distortion theorem (Lemma \ref{th31}) that there exists a constant $K$, depending only on the domain $B$, with 
$Kb_i\abs{z-w}\geq\abs{S_i(z)-S_i(w)}$
for all $z,w\in D$ and all $i$.
Because dim $X>s$ and the maps $S_i$ are contractions, we have $\lim_{p\to\infty}\sum_i (\diam S_i(D))^s=\infty$. So for $p$ large enough we obtain 
$$\sum_i b_i^s\geq\sum_i\sup_{z,w\in D}\frac{\abs{S_i(z)-S_i(w)}^s}{(K\abs{z-w})^s}\geq\sum_i\frac{(\diam S_i(D))^s}{(K\diam(D))^s}\to\infty,$$ which proves the proposition.
\end{proof}

The following simple lemma will be used in the proof of Theorem~\ref{thm1}.

\begin{lemma}\label{th24}
Let $h$ be a gauge function satisfying $h(2t)\leq K\, h(t)$ for some $K>0$ 
and all small $t$. Let $A\subset\R^n$ and $f\colon A\to \R^n$ be Lipschitz-continuous.
If $H_h(A)<\infty$, then $H_h(f(A))<\infty$.
\end{lemma}

\begin{proof}
Let $L$ be the Lipschitz constant of $f$; that is,
$\abs{f(z)-f(w)}\leq L\abs{z-w}$ for $z,w\in A$.
Choosing $m\in \N$ with $L\leq 2^m$ and putting $C=K^m$ we have
$h(Lt)\leq C\, h(t)$ for small~$t$.
Let $\delta>0$ and let $(B_j)$ be a $\delta$-cover of $A$.
Then $(f(B_j))$ is an $L\delta$-cover of $f(A)$
and 
$$\sum_j h(\diam f(B_j))\leq \sum_j h(L \diam B_j)
\leq C \sum_j h(\diam B_j)$$
so that $H_h(f(A))\leq C\,H_h(A)$.
\end{proof}

It follows 
that if $f\colon A\to f(A)$ is a bilipschitz map,
then $H_h(A)<\infty$ if and only if $H_h(f(A))<\infty$.

\section{Preliminaries for the proof of Theorem~\ref{thm1}} \label{sec3}
For a function $f\in B$ we choose $R>\abs{f(0)}$ 
such that
$\text{sing}(f^{-1})\subset D(0,R)$ and
put $G=\{z\colon \abs{f(z)}>R\}$, $H=\{z\colon \Rea z>\log R\}$ and $U=\exp^{-1}(G)$.
Eremenko and Lyubich \cite[section 2]{el} showed that there exists a holomorphic,
$2\pi i$-periodic
function $F\colon U\to H$ with $\exp(F(z))=f(\exp(z))$ for all $z\in U$.
We call $F$ the \emph{logarithmic transform} of~$f$. Moreover, 
they
showed that for every connected component $V$ of~$U$, the map
$F\vert_V\colon V\to H$ is bijective and the inverse map $\phi\colon H\to V$ satisfies
\begin{equation}\label{3a}
\abs{\phi'(w)}\leq\frac{4\pi}{\Rea w-\log R}
\end{equation}
for every $w\in H$. In terms of $F$ we obtain
\begin{equation}\label{3b}
\abs{F'(z)}\geq\frac{1}{4\pi}(\Rea F(z)-\log R)
\end{equation}
for all $z\in U$. In the sequel, we will assume without loss of generality that $R=1$. Let
$$\xi=\inf\{\Rea z\colon z\in U\}=\log\inf\{\abs{\zeta}\colon \abs{f(\zeta)}=R\}.$$
As in \cite{bks} we consider the function $h\colon (\xi,\infty)\to (0,\infty)$,
\begin{equation}\label{3c}
h(x)=\max_{\Rea z=x}\Rea F(z),
\end{equation}
and choose $z_x\in U$ with $\Rea z_x=x$ such that $h(x)=\Rea F(z_x)$. The function $h$ 
is increasing and convex
and
thus differentiable except
possibly at a countable set $C$ of $x$-values,
and for $x\in(\xi,\infty)\setminus C$ we have
\begin{equation}\label{3d}
h'(x)=F'(z_x),
\end{equation}
cf. \cite[p.~2562]{t}. Recalling that $R=1$ and 
$h(x)=\Rea F(z_x)$ we deduce from \eqref{3b} and \eqref{3d} that
\begin{equation}\label{3e}
\frac{h'(x)}{h(x)}\geq\frac{1}{4\pi}.
\end{equation}
Integration yields $\log h(x)\geq (x-\xi)/(4\pi)$ and thus
\begin{equation}\label{3e1}
h(x)\geq 2x
\end{equation}
for large $x$.

The following result is known as the Koebe distortion theorem and the Koebe one quarter theorem.
\begin{lemma}\label{th31}
Let $g\colon D(a,r)\to\C$ be a univalent function, $0<\lambda<1$ and 
$z\in\overline{D(a,\lambda r)}\setminus \{a\}$. Then
$$\frac{1}{(1+\lambda)^2}\abs{g'(a)}\leq\frac{\abs{g(z)-g(a)}}{\abs{z-a}}\leq\frac{1}{(1-\lambda)^2}\abs{g'(a)}$$
and
$$\frac{1-\lambda}{(1+\lambda)^3}\abs{g'(a)}\leq\abs{g'(z)}\leq\frac{1+\lambda}{(1-\lambda)^3}\abs{g'(a)}.$$
Moreover,
$$g(D(a,r))\supset D(g(a),\tfrac{1}{4}\abs{g'(a)}r).$$
\end{lemma}
Usually Koebe's theorems are stated only for the special case that $a=0$, 
$r=1$, $g(0)=0$ and $g'(0)=1$, but the above result follows easily from this special case.

For $a\in\C$ and $r>0$ we define the square
$$S(a,r)=\{z\in\C\colon \abs{\Rea(z-a)}<r,\; \abs{\Ima(z-a)}<r\}.$$
\begin{lemma}\label{th32}
For sufficiently large $x$ there exist $m\in\N$ and pairwise disjoint domains $W_1,W_2,\ldots,W_m\subset S(F(z_x),\frac{1}{4}\Rea F(z_x))$ satisfying
\begin{equation}\label{3f}
\sum_{j=1}^m\diam W_j \geq 10^{-5}\frac{h(x)^3}{h'(x)}
\end{equation}
such that $F^2$ is a univalent map from each $W_j$ onto $S(F(z_x),\frac{1}{4}\Rea F(z_x))$ whose inverse extends univalently to $S(F(z_x),\frac{1}{2}\Rea F(z_x))$.
\end{lemma}

\begin{proof}
Let $\phi$ be the branch of $F^{-1}$ which maps $F(z_x)$ to $z_x$ and let
$$
P_x=\phi(S(F(z_x),\tfrac{1}{4}\Rea F(z_x))).
$$
It follows from Koebe's theorem that
$$
D(z_x,r_x)\subset P_x\subset D(z_x,R_x)
$$
where
\begin{equation}\label{3i}
r_x=\frac{1}{4}\abs{\phi'(F(z_x))}\cdot\frac{1}{4}\Rea F(z_x)=\frac{1}{16}\frac{\Rea F(z_x)}{\abs{F'(z_x)}}=\frac{1}{16}\frac{h(x)}{h'(x)}
\end{equation}
and, by \eqref{3e},
\begin{equation}\label{3j}
R_x=4\abs{\phi'(F(z_x))}\cdot\frac{1}{4}\Rea F(z_x)=16r_x=\frac{h(x)}{h'(x)}\leq 4\pi
\end{equation}
for $x\notin C$. In particular, $P_x\subset H=\{z\in\C\colon \Rea z>0\}$ for large~$x$. For $k\in\Z$ we put
$$ 
d_{x,k}=\abs{\phi'(z_x+2\pi ik)},
\  v_{x,k}=\phi(z_x+2\pi ik)
\ \  \text{and}\  \ V_{x,k}=\phi(P_x+2\pi ik).
$$
Then, since $R=1$,
\begin{equation}\label{3l}
d_{x,k}\leq\frac{4\pi}{\Rea z_x}=\frac{4\pi}{x}\leq\frac{1}{4\pi}
\end{equation}
for large $x$
by \eqref{3a}. Again by Koebe's one quarter and distortion theorems,
\begin{equation}\label{3m}
D(v_{x,k},t_{x,k})\subset V_{x,k}\subset D(v_{x,k},T_{x,k})
\end{equation}
where
\begin{equation}\label{3n}
t_{x,k}=\frac{1}{4}d_{x,k}r_x
\end{equation}
and
$$
T_{x,k}=2d_{x,k}R_x=128t_{x,k}.
$$
By \eqref{3j} and \eqref{3l} we have
\begin{equation}\label{3p}
T_{x,k}\leq 2\frac{4\pi}{x}4\pi=\frac{32\pi^2}{x}\leq 1
\end{equation}
for large~$x$. Koebe's distortion theorem and \eqref{3l} also yield that
\begin{equation}\label{3q}
\abs{v_{x,k+1}-v_{x,k}}\leq 4\pi d_{x,k}\leq\frac{16\pi^2}{x}\leq 1
\end{equation}
for large~$x$. It is easy to see that $\Rea v_{x,k}\to\infty$ as $k\to\infty$ or $k\to -\infty$. On the other hand, for large $x$ there also exists $k_0\in\Z$ such that
$$\Rea v_{x,k_0}\leq x\leq\frac{1}{2}\Rea F(z_x).$$
Let
$$k_1=\max\left\{k\in\Z\colon \Rea v_{x,k}\leq\tfrac{3}{4}h(x)+1\right\}$$
and
$$k_2=\min\left\{k\in\Z\colon k\geq k_1\text{ and }\Rea v_{x,k}>\tfrac{5}{4}h(x)-1\right\}.$$
For $k_1<k<k_2$ we then have $\frac{3}{4}h(x)+1\leq\Rea
v_{x,k}\leq\frac{5}{4}h(x)-1$ and thus
$$
V_{x,k}\subset\left\{z\in\C\colon \tfrac{3}{4}h(x)\leq\Rea z\leq\tfrac{5}{4}h(x)\right\}
$$
by \eqref{3m} and \eqref{3p}. Now
\begin{equation}\label{3s}
\begin{aligned}
&\frac{1}{2}h(x)-2\leq\Rea v_{x,k_2}-\Rea v_{x,k_1} 
%\\ &
=\sum_{k=k_1}^{k_2-1}\left(\Rea v_{x,k+1}-\Rea v_{x,k}\right)
\\ &
\quad 
\leq\sum_{k=k_1}^{k_2-1}\abs{v_{x,k+1}-v_{x,k}} 
%\\ &
\leq4\pi\sum_{k=k_1}^{k_2-1}d_{x,k}
\leq4\pi\sum_{k=k_1+1}^{k_2-1}d_{x,k}+1
\end{aligned}
\end{equation}
by \eqref{3l} and \eqref{3q}. Hence
\begin{equation}\label{3t}
\begin{aligned}
\sum_{k=k_1+1}^{k_2-1}\diam(V_{x,k})&\geq 2\sum_{k=k_1+1}^{k_2-1}t_{x,k}=\frac{1}{2}r_x\sum_{k=k_1+1}^{k_2-1}d_{x,k}\\
&\geq\frac{1}{8\pi}r_x\left(\frac{1}{2}h(x)-3\right)\geq 
\frac{h(x)^2}{10^{3}h'(x)}
\end{aligned}
\end{equation}
by \eqref{3i}, \eqref{3m}, \eqref{3n} and \eqref{3s}. We now put
$$N_x=\left\lfloor\frac{1}{2\pi}\left(\frac{1}{2}h(x)-3\right)\right\rfloor.$$
Then
\begin{equation}\label{3u}
N_x\geq\frac{1}{20}h(x)
\end{equation}
for large~$x$. For $k_1< k< k_2$ there exists 
$l_k\in\Z$ such that 
if
$l_k\leq l < l_k+N_x$, then
$$v_{x,k}+2\pi il\in S(F(z_x),\tfrac{1}{4}\Rea F(z_x)-1)$$
and hence
$$V_{x,k}+2\pi il\subset S(F(z_x),\tfrac{1}{4}\Rea F(z_x))$$
by \eqref{3m} and \eqref{3p}. We now put $m=(k_2-k_1-1)N_x$
and denote by $W_1,\ldots,W_m$ the collection of the sets $V_{x,k}+2\pi il$
where $k_1< k< k_2$ and $l_k\leq l< l_k+N_x$. We deduce from \eqref{3t}
and \eqref{3u} that
$$
\sum_{j=1}^m\diam W_j =N_x\sum_{k=k_1+1}^{k_2-1}\diam V_{x,k}
\geq \frac{h(x)}{20}\cdot \frac{h(x)^2}{10^{3}h'(x)}
\geq 10^{-5}\frac{h(x)^3}{h'(x)}.
$$
To prove that the inverse function of $F^2\colon W_j\to S(F(z_x),\frac{1}{4}\Rea F(z_x))$ extends univalently to $S(F(z_x),\frac{1}{2}\Rea F(z_x))$ we only 
note that the argument showing that
$P_x\subset H$ for large $x$ actually yields that 
\[
\phi(S(F(z_x),\tfrac{1}{2}\Rea F(z_x)))\subset H
\]
 for large~$x$.
\end{proof}
The following growth lemma for real functions is well known, but for completeness
we include the short proof.
Usually it is stated for differentiable functions, but it also holds
for absolutely continuous functions. Note that such functions are
differentiable almost everywhere.
\begin{lemma}\label{th33}
Let $g\colon [x_0,\infty)\to\R$ be an increasing, absolutely continuous function
satisfying $\lim_{x\to\infty}g(x)=\infty$. Let $\delta>0$. Then there exists a 
measurable subset $E$ of $[x_0,\infty)$ 
having finite measure such that
$g'(x)\leq g(x)^{1+\delta}\text{ for }x\notin E$.
\end{lemma}
\begin{proof}
Suppose that $E= 
\{x\geq x_0\colon  g'(x)> g(x)^{1+\delta}\geq 1\}\neq \emptyset$.
Then
$$\int_E dx\leq \int_E \frac{g'(x)}{g(x)^{1+\delta}} dx\leq
\int_{x_1}^\infty  \frac{g'(x)}{g(x)^{1+\delta}} dx= 
\frac{1}{\delta g(x_1)^{\delta}}\leq\frac{1}{\delta}<\infty,$$
where $x_1=\inf E$. 
\end{proof}
We apply Lemma \ref{th33} to the function $h$
defined by \eqref{3c}.
Given $\eps>0$, we deduce from \eqref{3f} that,
in the situation of Lemma \ref{th32},
$$\sum_{j=1}^m\diam W_j\geq h(x)^{2-\eps}$$
for $x$ outside some set of finite
measure. We summarize the above results as follows.

\begin{lemma}\label{th34}
There exist $x_0>0$ and a subset $E$ of $[x_0,\infty)$ of
finite measure with the following property: 
if $x\in[x_0,\infty)\setminus E$, then there exist $m\in\N$ and pairwise
disjoint sets 
$W_1,...,W_m\subset S(F(z_x),\tfrac{1}{4}\Rea F(z_x))$
satisfying
\begin{equation}\label{3v}
\sum_{j=1}^m\diam W_j \geq h(x)^{3/2}.
\end{equation}
such that
$F^2\colon W_j\to S(F(z_x),\tfrac{1}{4}\Rea F(z_x))$
is a conformal map whose inverse extends univalently to the
square $S(F(z_x),\frac{1}{2}\Rea F(z_x))$.
\end{lemma}

\section{Proof of Theorem~\ref{thm1}} \label{sec4}

By \eqref{1a} we have $h(t)=t^{1+\eps(t)}$ where
$\eps(t)\to 0$ as $t\to 0$. Similarly as at the beginning of the proof
of Lemma \ref{lemma1}, we can assume without loss of generality that $\eps$
is positive and non-decreasing. Let $F\colon U\to H$ be the logarithmic transform
of $f$ as defined in section~\ref{sec3}. 
The set 
corresponding to $\Esc(f,(p_n))$ is the set
$$Z=\{z\in U \colon \Rea F^n(z)\to\infty,
\Rea F^n(z)\leq \log p_n \text{ for large }n\}.$$
Note that $\exp Z\subset \Esc(f,(p_n))$.
We will show first that there exists a bounded subset $Y$ of $Z$
satisfying $H_h(Y)=\infty$. 
From this we will then deduce that
$H_h(\Esc(f,(p_n)))=\infty$.
Our main tool in the proof will be Lemma \ref{th23}, so we have to construct a sequence of iterated function schemes.

Let $x_0$ and $E$ be as in Lemma \ref{th34}.
Since $E$ has finite measure, there exists 
$M\geq x_0$ with $\text{meas}(E\cap[M,\infty))<\frac{1}{2}$.
So we can find a sequence $(x_k)$ with $x_1\geq M$, $x_1\notin E$ and
$$x_k\in[h(x_{k-1}),h(x_{k-1})+1]\setminus E\quad \text{for }k\geq 2.$$
It follows from \eqref{3e1}  that $x_k\to\infty$ if $x_1$ is chosen large enough.
Let
$$S_k=S(F(z_{x_k}),\tfrac{1}{4}\Rea F(z_{x_k}))
=S(F(z_{x_k}),\tfrac{1}{4} h(x_k)).$$
Let $W_{k,1},\ldots,W_{k,l_k}$ be the sets obtained from Lemma \ref{th34}. The maps
$$\tilde{P}_{k,j}=(F^2\vert_{W_{k,j}})^{-1}$$
define an iterated function scheme $\tilde{P}_k$ on $S_k$. By conjugating $\tilde{P}_{k,j}$ with the affine map $L_k$ that sends $S_k$ to
the square $D$ defined by~\eqref{defD}
we obtain an iterated function scheme $P_k=\{P_{k,1},\dots,P_{k,l_k}\}$ on~$D$. 
In other words, we define $P_{k,j}=L_k\circ \tilde{P}_{k,j}\circ L_k^{-1}$ 
with $L_k\colon S_k\to D$, 
$$L_k(z)=\frac{2}{h(x_k)}(z-F(z_{x_k}))+\frac12 +\frac{i}{2}.$$
Further, we have maps $\tilde{Q}_{k,j}\colon S_{k+1}\to V_{k,j}\subset S_k$ which are inverse branches of $F\vert_{V_{k,j}}$, for $1\leq j\leq m_k$. Setting
$$Q_{k,j}=L_k\circ\tilde{Q}_{k,j}\circ L_{k+1}^{-1}$$
defines an iterated function scheme $Q_k$ on~$D$.
It is obvious that condition~(1) from Lemma \ref{th23} holds.
It remains to verify condition~(2). Since all the maps
$\tilde{P}_{k,i}$ and $\tilde{Q}_{k,j}$ can be continued univalently to a square
with twice the side length of $S_k$ resp. $S_{k+1}$, the existence of 
positive numbers
$\tilde{b}_{k,i}$ and $\overline{b}_{k,j}$ 
with
$\tilde{b}_{k,i}\abs{z-w}\leq\abs{P_{k,i}(z)-P_{k,i}(w)}$ and $\overline{b}_{k,j}\abs{z-w}\leq\abs{Q_{k,j}(z)-Q_{k,j}(w)}$
follows immediately from the
Koebe distortion theorem (Lemma \ref{th31}). For the proof of \eqref{2a}, we
will use \eqref{3v}. First note that, again by the Koebe distortion
theorem, there exists an absolute constant $K>0$
with
$$K\tilde{b}_{k,i}\geq\frac{\abs{P_{k,i}(z)-P_{k,i}(w)}}{\abs{z-w}}\geq\tilde{b}_{k,i}.$$
Choosing $z_0,w_0\in D$ with
$$\abs{P_{k,i}(z_0)-P_{k,i}(w_0)}=\max_{z,w\in D}\abs{P_{k,i}(z)-P_{k,i}(w)}
=\frac{2\diam W_{k,i}}{h(x_k)},$$
we obtain
\begin{align*}
K\tilde{b}_{k,i}&\geq\sup_{z,w\in D}\frac{\abs{P_{k,i}(z)-P_{k,i}(w)}}{\abs{z-w}}\geq\frac{\abs{P_{k,i}(z_0)-P_{k,i}(w_0)}}{\abs{z_0-w_0}}\\
&\geq\frac{\abs{P_{k,i}(z_0)-P_{k,i}(w_0)}}{\sup_{z,w\in D}\abs{z-w}}
=\sqrt{2}\frac{\diam W_{k,i}}{h(x_k)},
\end{align*}
so
$$\tilde{b}_{k,i}\geq C\frac{\diam W_{k,i}}{h(x_k)}$$
for some absolute constant~$C$. Using \eqref{3v}, we obtain
$$\sum_{i=1}^{l_k}\tilde{b}_{k,i}\geq 
C\frac{\sum_{i=1}^{l_k}\diam W_{k,i}}{h(x_k)}\geq C\, h(x_k)^{1/2}>1$$
if $x_1$ was chosen large enough at the beginning. 
By Lemma \ref{th23}, there exists an increasing sequence $(n_i)$ such that
the limit set $X$ of $(T_k)$, where $T_k$ is defined as in \eqref{2b},
satisfies $H_h(X)=\infty$. 

Put $Y=L_1^{-1}(X)$.
By increasing $(n_i)$ if necessary, we can
achieve that $\Rea F^k(z)\leq\log p_k$ if $z\in Y$ and $k$
is large enough. 
We also have $\Rea F^n(z)\to \infty$ as $n\to\infty$ for $z\in Y$.
To see this, 
let $z=L_1^{-1}(w)$, where $w\in X$. 
Let $k\in\N$ and put $i_k=\max\{i\colon n_i\leq k\}$. Then
$$(T_{k,j_k}^{-1}\circ\ldots\circ T_{1,j_1}^{-1})(w)=L_{i_k+1}(F^{2k-i_k}(L_1^{-1}(w)))=L_{i_k+1}(F^{2k-i_k}(z)).$$
Since $i_k\to\infty$ and $\max_{z\in S_k}\Rea z\to\infty$ as
$k\to\infty$, 
we can deduce from this
that $\Rea F^{2k-i_k}(z)\to\infty$ as $k\to\infty$,
from which we can easily deduce that $\Rea F^n(z)\to \infty$ as $n\to\infty$.
Altogether we thus have $Y\subset Z$.

Since $H_h(X)=\infty$ and infinite $H_h$-measure is invariant
under affine mappings for any gauge function~$h$, we have $H_h(Y)=\infty$. 
In order to deduce that $H_h(\Esc(f,(p_n)))=\infty$ we use
Lemma~\ref{th24}.
Recall that 
\begin{equation}\label{4b1}
\exp Y\subset \exp Z\subset \Esc(f,(p_n)).
\end{equation}
Since $Y$ is bounded and $H_h(Y)=\infty$, there exists $y_0\in\R$ with
\begin{equation}\label{4b2}
H_h(Y\cap\{z\colon \Ima z\in(y_0,y_0+\pi)\})=\infty.
\end{equation}
Noting that $\eps$ is non-decreasing, we also see that
$$h(2t)=2 t\cdot  (2t)^{\eps(2t)}\leq 3t\cdot t^{\eps(2t)}
\leq 3t\cdot t^{\eps(t)}=3 h(t)$$
for small~$t$.
Since $\exp$ restricted to $Y\cap\{z\colon \Ima  z\in(y_0,y_0+\pi)\}$
is a bilipschitz mapping, Lemma~\ref{th24} and \eqref{4b2} yield
$$H_h(\exp(Y\cap\{z\colon \Ima  z\in(y_0,y_0+\pi)\}))=\infty.$$
An application of \eqref{4b1} finishes the proof.\qed
\begin{rem}
The result of~\cite{bkz} yields a sequence $(S_k)$ of squares tending 
to $\infty$ and associated iterated function schemes $\tilde{P}_k$ 
as in the above proof. Lemma~\ref{th34} yields 
additional information about the
``density'' of such squares. However, this is not essential for
the argument, since otherwise we could replace the $\tilde{Q}_{k,j}$ by 
inverse branches of some iterate of~$F$.
\end{rem}

\section{Preliminaries for the proof of Theorem~\ref{thm2}} \label{sec5}
The function $f$ will have the property that it is bounded outside a
narrow strip. There is a well-established technique to construct such
functions using contour integrals, 
cf. \cite[Chapter III, Problems 158-160]{PolyaSzegoe},
 \cite{rrrs}, \cite{s4} and, in particular, \cite{Rempe}.
In order to apply this method
we need some estimates concerning conformal mappings of strips.
Let $\Omega$ be a domain of the form 
$\Omega=\{x+iy\colon  |y|<\phi(x)\}$ with some non-negative function
$\phi\colon \R\to\R$ and let 
$w\colon \Omega\to \{x+iy\colon  |y|<\pi\}$
be a conformal map satisfying $w(x)\to\pm\infty$
as $x\to\pm\infty$. Put
\[
\overline{H}(x)=\sup_{|y|<\phi(x)}\Rea w(x+iy)
\quad\text{and}\quad
\underline{H}(x)=\inf_{|y|<\phi(x)}\Rea w(x+iy).
\]
The celebrated Ahlfors distortion theorem~\cite[\S 2]{Ahlfors30},
specialized to strips of the above form, says that 
\[
\underline{H}(x_2)-\overline{H}(x_1)\geq 
\pi\int_{x_1}^{x_2} \frac{dx}{\phi(x)} -8\pi
\quad\text{if}\ 
\int_{x_1}^{x_2} \frac{dx}{\phi(x)}>4.
\]
We mention that Ahlfors denoted the ``width'' of a cross section 
(of more general strips) by $\theta(x)$. 
In our setting we have $\theta(x)=2\phi(x)$.

Ahlfors~\cite[\S 3]{Ahlfors30} also proved an inequality 
in the opposite direction, provided that $\phi$ satisfies some
regularity conditions. (It is easy to see that some
additional hypotheses are necessary for such estimates.)

Suppose that $\phi$ is bounded, continuous and of bounded variation 
on every finite interval. 
Following Ahlfors we
denote by $\phi_m(x_1,x_2)$ the minimum of $\phi$ 
and by $V(x_1,x_2)$ the total variation of $\phi^2$
in the interval $[x_1,x_2]$.
Noting that $w$ extends continuously to $\partial G$, with
$\partial G$ being mapped bijectively onto $\{x+iy\colon |y|=\pi\}$, we put
\[
\overline{x} =\Rea w^{-1} (\overline{H}(x)+i\pi)
\quad\text{and}\quad
\underline{x} =\Rea w^{-1} (\underline{H}(x)+i\pi).
\]
With
$L=\sup_{x\in\R} \phi(x)$
Ahlfors's result~\cite[p.~15]{Ahlfors30} then takes the form
\[
\overline{H}(x_2)-\underline{H}(x_1)\leq 
\pi\int_{\underline{x}_1}^{\overline{x}_2} \frac{dx}{\phi(x)} +8\pi
L^2 \frac{\phi_m(\underline{x}_1,\overline{x}_2)^2 + V(\underline{x}_1,\overline{x}_2)}{\phi_m(\underline{x}_1,\overline{x}_2)^4}.
\]
Suppose in addition that $\phi$ is decreasing. Then this simplifies to 
\begin{equation}\label{5new1}
\begin{aligned}
\overline{H}(x_2)-\underline{H}(x_1)
&\leq
\pi\int_{\underline{x}_1}^{\overline{x}_2} \frac{dx}{\phi(x)} +8\pi
L^2\frac{\phi(\underline{x}_1)^2}{\phi(\overline{x}_2)^4} 
\\ &
\leq
\pi\int_{\underline{x}_1}^{\overline{x}_2} \frac{dx}{\phi(x)} +
\frac{8\pi L^4}{\phi(\overline{x}_2)^4}.
\end{aligned}
\end{equation}
Ahlfors~\cite[p.~15]{Ahlfors30} also showed that 
\[
\int_{\underline{x}_1}^{x_1} \frac{dx}{\phi(x)} \leq 8
\quad\text{and}\quad
\int_{x_2}^{\overline{x}_2} \frac{dx}{\phi(x)} \leq 8.
\]
This implies that $8\geq (\overline{x}_2-x_2)/\phi(x_2)$ and hence
\begin{equation}\label{5new2}
\overline{x}_2\leq x_2+8 \phi(x_2).
\end{equation}

We will 
also assume that
$\phi(x)\leq 1/x$ for large~$x$.
Assuming that $\underline{x}_1\geq 0$  we can now deduce from~\eqref{5new1} 
and~\eqref{5new2} that
\[
\overline{H}(x_2)-\underline{H}(x_1)
\leq \pi \frac{\overline{x}_2}{\phi(\overline{x}_2)}
+\frac{8\pi L^4}{\phi(\overline{x}_2)^4}
\leq \frac{9\pi L^4}{\phi(\overline{x}_2)^4}
\leq \frac{9\pi L^4}{\phi(x_2+8 \phi(x_2))^4}
\]
for large~$x_2$. For us  only the case where $x_1$ is fixed and $x=x_2\to\infty$
is of interest. 
We obtain the following result.
\begin{lemma}\label{ahlf1}
Let $\phi\colon \R\to\R$ be a positive, bounded, continuous and
 decreasing function satisfying $\phi(x)\leq 1/x$ for large~$x$.
Let 
\[
w\colon \{x+iy\colon  |y|<\phi(x)\}\to \{x+iy\colon  |y|<\pi\}
\]
be a conformal map satisfying $w(x)\to\pm\infty$
as $x\to\pm\infty$. Then there exists a constant $C$ such that
\[
\sup_{|y|<\phi(x)}\Rea w(x+iy)\leq \frac{C}{\phi(x+8 \phi(x))^4}
\]
for all large $x$.
\end{lemma}
Let $\phi$ be as above and put
\begin{equation}\label{5new3}
S=\{x+iy\colon  x>0, |y|<\phi(x)\}.
\end{equation}
The method of contour integrals described in the papers  mentioned above consists of 
defining a function $f$ by 
\[
f(z)=\frac{1}{2\pi i}\int_{\partial S}
\frac{\exp\left(e^{w(\zeta)}\right)}{\zeta-z}d\zeta
\]
for $z\in\C\backslash\overline{S}$ and analytic continuation
of $f$ to the whole plane by deforming the path of integration.
The results of Rempe~\cite[Theorem~1.7]{Rempe} imply the following lemma.
\begin{lemma}\label{remp1}
Let $\phi$ and $w$ be as in Lemma~\ref{ahlf1} and define $S$ by~\eqref{5new3}.
Then there exists $f\in B$ satisfying
\[
f(z)=\begin{cases}
\exp\left(e^{w(z)}\right)+O(1/z)&\text{for }z\in S,\\
O(1/z)&\text{for }z\notin S .
\end{cases}
\]
\end{lemma}
It follows from Lemma~\ref{ahlf1}  that the function $f$ 
in Lemma~\ref{remp1} satisfies
\begin{equation}\label{5new4}
|f(z)|\leq \exp \left( \exp \left(
\frac{C}{\phi(|z|+8 \phi(|z|))^4}\right)\right)
\end{equation}
for some $C>0$ if $|z|$ is large.
Replacing $f$ by $\eps f$ with a small constant $\eps$ we may 
assume in addition that
\begin{equation}\label{5i}
\text{sing}\left(f^{-1}\right)\subset\left\{z\in\C\colon \abs{z}<\tfrac{1}{2}\right\},
\end{equation}
\begin{equation}\label{5j}
\abs{f(z)}\leq\tfrac{1}{2}\text{ for }\abs{z}\leq 1
\end{equation}
and
\begin{equation}\label{5k}
\abs{f(z)}\leq 1\text{ for }z\notin S.
\end{equation}
It is apparent from the above discussion that the behavior
of $\phi(x)$ as $x\to-\infty$ is irrelevant for our purposes.
In fact, it suffices to define the function $\phi$ on an
interval $[x_0,\infty)$, as it can be continued to $\R$
by setting $\phi(x)=\phi(x_0)$ for $x<x_0$.
We shall use the following result to define the function~$\phi$.
\begin{lemma}\label{th53}
Let $\alpha\colon [x_0,\infty)\to(0,\infty)$ be decreasing and
continuous and let $\beta\colon (0,1]\to (0,\infty)$ be increasing and continuous.
Then there exists a decreasing, continuous function 
$\phi\colon [x_0,\infty)\to(0,1]$ satisfying 
$\phi(x+\alpha(x))\leq\beta(\phi(x))$
for large~$x$.
\end{lemma}
\begin{proof}
We may assume that 
the function $\sigma$ given by 
$\sigma(x)=x+\alpha(x)$ is
increasing as this can achieved by replacing $\alpha$ by a decreasing,
continuous function $\alpha^*\colon [x_0,\infty)\to(0,\infty)$  satisfying
$\alpha^*(x)\leq\alpha(x)$ for all~$x$.
Similarly, we may assume that $\beta(x)<x$ for all~$x$.
We now define $\phi$ in the interval $[x_0,\sigma(x_0)]$ by
$\phi(x_0)=1$, $\phi(\sigma(x_0))=\beta(\phi(x_0))=\beta(1)$,
and linear interpolation in $(x_0,\sigma(x_0))$.
For $k\in\N$ we extend this to the interval $(\sigma^k(x_0),\sigma^{k+1}(x_0)]$
by putting $\phi(\sigma^k(x))=\beta^k(\phi(x))$. Since
$\sigma^k(x_0)\to\infty$ as $k\to\infty$, this defines a decreasing,
continuous function $\phi\colon [x_0,\infty)\to(0,1]$ satisfying
$\phi(x+\alpha(x))= \beta(\phi(x))$.
\end{proof}
We will also use the following result known as the Besicovich covering lemma \cite[Theorem 3.2.1]{d}. Here $B(x,r)$ is the open ball of radius $r$ around a point $x\in\R^n$.
\begin{lemma}\label{th57}
Let $K\subset\R^n$ be bounded and $r\colon K\to(0,\infty)$. Then there exists an at most countable subset $L$ of $K$ satisfying
$$K\subset\bigcup_{x\in L}B(x,r(x))$$
such that no point in $\R^n$ is contained in more than $4^{2n}$ of the 
balls $B(x,r(x))$, $x\in L$.
\end{lemma}
The following result is a simple consequence of the Besicovich covering lemma; 
see~\cite[Lemma 5.2]{Bergweiler10} for a similar result concerning Hausdorff dimension.
\begin{lemma}\label{th58}
Let $K\subset\R^n$ and let $h$ be a gauge function. Suppose that for all
$x\in K$ and $\eps>0$ there exists $\delta(x)\in(0,\eps)$, $N(x)\in\N$ and balls
$B_1,\ldots,B_{N(x)}$ such that
$$K\cap B(x,\delta(x))\subset\bigcup_{j=1}^{N(x)} B_j
\quad\text{and}\quad
\sum_{j=1}^{N(x)} h(\diam B_j)\leq\eps \delta(x)^n.$$
Then $H_h(K)=0$.
\end{lemma}
\begin{proof}
We may assume that $K$ is bounded, say $K\subset B(0,R)$. Let $\eps>0$.
For $x\in K$, let $\delta(x)\in(0,\eps),N(x)\in\N$ and $B_1(x),\ldots,B_{N(x)}(x)$
be as given in the hypothesis. Let $L$ be as in Lemma \ref{th57}. Then
$$\{B_j(x)\colon x\in L,1\leq j\leq N(x)\}$$
is an open cover of~$K$. We may assume that
$\diam B_j(x)\leq 2\delta(x)<2\eps$ for $x\in L$ and $1\leq j\leq N(x)$. Moreover,
$$\sum_{x\in L}\sum_{j=1}^{N(x)}h(\diam B_j(x))
\leq\eps\sum_{x\in L}\delta(x)^n\leq\eps \, 4^{2n}(R+\eps)^n\omega_n,$$
where $\omega_n$ is the volume of the unit ball in $\R^n$.
Thus
$$H_h(K)\leq\eps \, 4^{2n}(R+\eps)^n\omega_n,$$
and the conclusion follows.
\end{proof}

\section{Proof of Theorem~\ref{thm2}} \label{sec6}
Let $(p_n)$ and $h$ be as in the hypothesis.
First we note that if $p_n\geq q_n$ for large~$n$, then 
$\Unb(f,(p_n))\subset \Unb(f,(q_n))$ and thus
$$H_h(\Unb(f,(p_n)))\leq H_h(\Unb(f,(q_n))).$$
 Hence it is no loss of generality to assume that
\begin{equation}\label{6a}
p_n\leq n\quad \text{and}\quad p_n-p_{n-1}\geq\frac{6}{n^2}
\end{equation}
for large~$n$, since otherwise we could pass to the sequence $(q_n)$ defined by
$$
q_n=\min\big\{n,\inf_{k\geq n}p_k\big\}+6\sum_{k=1}^n\frac{1}{k^2}-\pi^2,
$$
which has the above properties and satisfies $q_n\leq p_n$.

We write our gauge function $h$ in the form $h(t)=tg(t)$. Then \eqref{1b}
says that
$g(t)\to 0$ as $t\to 0$.
By
\eqref{2j} we may assume that $g$ is increasing and
satisfies $g(t)\geq t$ for all~$t$, since otherwise we 
could replace $g(t)$ by $t+\sup_{s\leq t}g(s)$.

We consider the function $\tau\colon (0,1]\to(0,\infty),$
$$\tau(t)=\left(\frac{t}{4}\exp\left(-\exp\left(\frac{1}{t^5}\right)\right)\right)^{1/t}.$$
We apply Lemma \ref{th53} to the function $\beta=g^{-1}\circ\tau$ and to a
decreasing function $\alpha\colon [p_1,\infty)\to(0,1/4]$ satisfying
$\alpha(x)\leq 1/n^2$ for $x\geq p_{n-1}$ and $n\geq 2$.
For example, we may define $\alpha$ by putting $\alpha(p_{n-1})=1/n^2$ and interpolating
linearly in the intervals $[p_{n-1},p_n]$ for $ n\geq 2$.
We now choose $\phi\colon [p_1,\infty)\to(0,1]$ according to Lemma \ref{th53}
and obtain
\begin{equation}\label{6d}
g\left(\phi\left(x+\frac{1}{n^2}\right)\right)\leq\tau(\phi(x))\quad\text{for }x\geq p_{n-1}.
\end{equation}
Since $g(t)\geq t$ this implies, together with \eqref{6a}, that
$$\phi(p_n)\leq g(\phi(p_n))\leq\tau(\phi(p_{n-1}))\leq\frac{1}{4}\phi(p_{n-1})$$
and thus
\begin{equation}\label{6e}
\phi(p_n)\leq\frac{1}{4^n}
\end{equation}
by induction. 
Combined with~\eqref{6a} this yields
$\phi(n)\leq 4^{-n}\leq (n+1)^{-2}$
so that
\begin{equation}\label{6h1}
\phi(x)\leq\frac{1}{x^2}\leq\frac{1}{x}
\end{equation}
for $x\geq p_1$. 
Thus $\phi$ satisfies the hypotheses of Lemma~\ref{ahlf1} and
we may choose $f\in B$ according to Lemma~\ref{remp1}.
As mentioned after Lemma~\ref{remp1}, we may assume that \eqref{5new4}--\eqref{5k} hold.

Fix $\xi\in \Unb(f,(p_n))$ and $\eps>0$. In order to apply Lemma \ref{th58},
we have to show that there exist $\delta(\xi)\in (0,\eps)$ and a positive integer
$N(\xi)$ such that $D(\xi,\delta(\xi))\cap \Unb(f,(p_n))$ can be covered by
$N(\xi)$ disks $D_1,\ldots,D_{N(\xi)}$ such that
\begin{equation}\label{6f}
\sum_{j=1}^{N(\xi)}h(\diam D_j)\leq\eps \, \delta(\xi)^2.
\end{equation}
In order to show that such $\delta(\xi)$ and $N(\xi)$ exist we note first that there exist arbitrarily large $n$ with
\begin{equation}\label{6g}
\abs{f^n(\xi)}\geq p_n\quad \text{and}\quad 
\abs{f^n(\xi)}\geq\frac{6}{n^2}+\max_{0\leq k\leq n-1}\abs{f^k(\xi)}.
\end{equation}
Indeed, since $(\abs{f^m(\xi)})_{m\in\N}$ is unbounded, there exist arbitrarily large $m$ such that
$$\abs{f^m(\xi)}\geq\frac{6}{m^2}+\max_{0\leq k\leq m-1}\abs{f^k(\xi)}.$$
For $m$ with this property we take
$$n=\min\{k\in\N\colon k\geq m\ \text{and}\ \abs{f^k(\xi)}\geq p_k\}.$$
Then $n$ satisfies \eqref{6g}.
For such $n$ 
we put
\begin{equation}\label{6g1}
t_n=\abs{f^n(\xi)}-\frac{1}{n^2},
\quad
s_n=\abs{f^n(\xi)}-\frac{3}{n^2}
\quad \text{and}\quad 
r_n=\abs{f^n(\xi)}-\frac{5}{n^2}.
\end{equation}
Then
$$t_n>s_n>r_n\geq p_n-\frac{5}{n^2}>p_{n-1}$$
and thus
\begin{equation}\label{6g2}
g(\phi(t_n))\leq\tau(\phi(s_n))
\end{equation}
by \eqref{6d}. Moreover, \eqref{6e} yields
\begin{equation}\label{6h}
\phi(t_n)\leq\phi(s_n)\leq\phi(r_n)
\leq\frac{1}{4^{n-1}}\leq\frac{1}{n^2}\leq\frac{1}{n},
\end{equation}
Put
$$l_n=t_n-s_n=\frac{2}{n^2}$$
and let 
$Q_n=S(f^n(\xi),1/n^2)$ 
be the square of side length $l_n$ centered at $f^n(\xi)$.
We can cover $Q_n\cap S$ by $N_n$ squares of side length $\phi(t_n)$, 
where $S$ is given by~\eqref{5new3} and
\begin{equation}\label{6i}
N_n\leq\frac{2l_n}{\phi(t_n)}.
\end{equation}
Put
$$
\rho_n=\frac{1}{\abs{(f^n)'(\xi)}}.
$$
We deduce from \eqref{5i} and \eqref{5j} that, for large~$n$,
the branch of the inverse function of $f^n$ which maps $f^n(\xi)$ to $\xi$
extends univalently to a square centered at $f^n(\xi)$ which has twice the
side length of $Q_n$. Koebe's distortion theorem now shows that, for certain
absolute constants $c_1$ and $c_2$, which in fact could be determined explicitly,
we can cover $D(\xi,c_1\rho_nl_n)\cap f^{-n}(S)$ by $N_n$ disks $D_1,\ldots,D_{N_n}$
of diameter at most $c_2\rho_n\phi(t_n)$.
Using \eqref{5j} and \eqref{5k} we see that
$D(\xi,c_1\rho_nl_n)\cap \Unb(f,(p_k))$ is covered by these disks.
We have
\begin{equation}\label{6k}
\begin{aligned}
\sum_{j=1}^{N_n}h(\diam D_j)&\leq N_nh(c_2\rho_n\phi(t_n)) 
%\\ &
=N_nc_2\rho_n\phi(t_n)g(c_2\rho_n\phi(t_n)) 
\\ &
\leq 2c_2\rho_n l_n g(c_2\rho_n\phi(t_n))
\end{aligned}
\end{equation}
by \eqref{6i}. 

Put
$\delta_n=c_1\rho_nl_n$.
It is easy to see that $\rho_n\to 0$ and thus $\delta_n\to 0$. We will show that \eqref{6f} holds for $\delta(\xi)=\delta_n$ and $N(\xi)=N_n$ if $n$ is sufficiently large. 
In order to do so we note that, by \eqref{6k}, it suffices to show that
$$2c_2\rho_n l_n g(c_2\rho(n)\phi(t_n))\leq\eps \, \delta_n^2,$$
which, by the definition of $\delta_n$ and $l_n$,  is equivalent to
$$g(c_2\rho_n\phi(t_n))\leq\frac{\eps \delta_n^2}{2c_2l_n\rho_n}=\frac{c_1^2\eps}{c_2n^2}\rho_n.$$
Since $\rho_n\to 0$ we have $c_2\rho_n\leq 1$ for large $n$ and thus it suffices to show that
\begin{equation}\label{6l}
g(\phi(t_n))\leq\frac{c_1^2\eps}{c_2n^2}\rho_n
\end{equation}
for large~$n$. To estimate $\rho_n$ we note that
$$\rho_n=\frac{1}{\prod_{k=0}^{n-1}\abs{f'(f^k(\xi))}}$$
by the chain rule. Cauchy's integral formula implies that
$$
\max_{\abs{z}\leq r}\abs{f'(z)}\leq\frac{1}{R-r}\max_{\abs{z}\leq R}\abs{f(z)}
$$
for $0<r<R$. 
Since
$$\abs{f^k(\xi)}\leq \abs{f^n(\xi)}-\frac{6}{n^2}
\leq r_n-\frac{1}{n^2}\quad \text{for }0\leq k\leq n-1$$
by \eqref{6g} and \eqref{6g1},
we deduce from~\eqref{5new4} that
\[
\abs{f'(f^k(\xi))}\leq n^2\max_{\abs{z}=r_n}\abs{f(z)}
\leq n^2\exp\left(\exp\Big(\frac{C}{\phi(r_n+8\phi(r_n))^4}\Big)\right).
\]
Now~\eqref{6h} says that $n^2\leq 1/\phi(s_n)$ and~\eqref{6h}
also yields that 
\[
8\phi(r_n)\leq \frac{8}{4^{n-1}}\leq \frac{2}{n^2}=s_n-r_n
\]
and hence $r_n+8\phi(r_n)\leq s_n$ for large~$n$.
We conclude that
\begin{align*}
\abs{f'(f^k(\xi))}
&\leq\frac{1}{\phi(s_n)}\exp\left(\exp\Big(\frac{C}{\phi(s_n)^4}\Big)\right)
%\\ &
\leq\frac{1}{\phi(s_n)}\exp\left(\exp\Big(\frac{1}{\phi(s_n)^5}\Big)\right),
\end{align*}
provided $n$ is sufficiently large. 
We obtain
\begin{align*}
\rho_n
&\geq
\left(\phi(s_n) \exp\left(-\exp\Big(\frac{1}{\phi(s_n)^5}\Big)\right)\right)^n
\\ &
=4^n\left(\frac{\phi(s_n)}{4}
\exp\left(-\exp\Big(\frac{1}{\phi(s_n)^5}\Big)\right)\right)^n.
\end{align*}
Using~\eqref{6h} we obtain
$\rho_n 
\geq 4^n\tau(\phi(s_n))
\geq 4^ng(\phi(t_n))$ 
which implies~\eqref{6l} for large~$n$.

\begin{acknowledgement}
We thank Lasse Rempe and the
referee for a great number of very helpful comments and suggestions.
\end{acknowledgement}

\end{document}